\documentclass{amsart}
\usepackage{amsfonts,amssymb,amsmath,amsthm}
\usepackage{url}
\usepackage{enumerate}

\newtheorem{thrm}{Theorem}[section]
\newtheorem{lem}[thrm]{Lemma}
\newtheorem{prop}[thrm]{Proposition}
\newtheorem{cor}[thrm]{Corollary}
\theoremstyle{definition}
\newtheorem{definition}[thrm]{Definition}
\newtheorem{remark}[thrm]{Remark}
\numberwithin{equation}{section}

\author{Anatolij~Plichko}

\address{
Department of Mathematics\\
Cracow University of Technology\\
Cracow, Poland} \email{aplichko@pk.edu.pl}


\begin{document}
\title[On uniform continuity of convex bodies]
{On uniform continuity of convex bodies with respect to
measures in Banach spaces}
\begin{abstract}
Let $\mu$ be a probability measure on a separable Banach space
$X$. A subset $U\subset X$ is $\mu$-continuous if $\mu(\partial
U)=0$. In the paper the $\mu$-continuity and uniform
$\mu$-continuity of convex bodies in $X$, especially of balls and
half-spaces, is considered. The $\mu$-continuity is interesting
for study of the Glivenko-Cantelli theorem in Banach spaces.
Answer to a question of F.~Tops{\o}e is given.
\end{abstract}
\maketitle

\section{Introduction}

We consider real {\it separable} Banach spaces $X$ only, and do
not mention this explicitly. By $B_r(z)$ we denote the closed ball
of $X$ with radius $r$ and center $z$, by $\,S_X$ the unit sphere
of $X$, by $\,X^*$ the dual of $X$. Given subset $U\subset X$ and
$\delta>0\,$, $\,\partial_\delta U\;$ stands for the
$\delta$-neighborhood of its boundary $\partial U$. All measures
are assumed to be {\it probability} measures and are defined on
Borel subsets of $X$. Let measure $\mu$ be defined on a Banach
space $X$. Our purpose is to analyze the following concepts.

\begin{definition}
A subset $U$ of a Banach space $X$ is called $\mu$-{\it
continuous} if $\mu(\partial U)=0$. A class $\mathcal{U}$ of
subsets in $X$ is called $\mu$-{\it continuous} if each set
$U\in\mathcal{U}$ is $\mu$-continuous \cite[p.~149]{Topsoe3}.

We call a class $\mathcal{U}$ {\it uniformly $\mu$-continuous} if
for every $\varepsilon>0$ there is a $\delta>0$ such that
$\mu(\partial_\delta U)<\varepsilon$ for all $U\in\mathcal{U}$. We
call a class $\mathcal{U}$ {\it uniformly discontinuous} if it is
not $\mu$-uniformly continuous with respect to any measure $\mu$.

A Banach space $X$ is said to be $\mathcal{U}$-{\it ideal}$\,$
if $\mu$-continuity of $\mathcal{U}$ implies its uniform
$\mu$-continuity, for any measure $\mu$ \cite[p.~283]{Topsoe1}.
\end{definition}

These notions have their origin in the theory of empirical
distributions and are connected with generalizations of the
Glivenko-Cantelli theorem to metric spaces \cite{Sazonov},
\cite{BillingsleyTopsoe}, \cite{Topsoe1}, \cite{Topsoe2},
\cite{Topsoe3}, \cite{TopsoeDudleyHoffmannJorgensen},
\cite{Szabados}, \cite{MatsakPlichko}. Sometimes (see e.g.
\cite[p.~279]{Topsoe1}) one talks about the $U$-continuity of a
measure $\mu$ instead of the $\mu$-continuity of sets $U$. Uniform
$\mu$-continuity was repeatedly used without a name (see e.g.
\cite{Sazonov}, \cite[p.~2]{BillingsleyTopsoe},
\cite[p.~282]{Topsoe1}, \cite[p.~151]{Topsoe3}), and is equivalent
to so-called $\mu$-uniformity, which we do not consider. Of
course, every uniformly $\mu$-continuous class is
$\mu$-continuous. Examples showing that the opposite statement is
false are well known. We present some such examples below.

Of course, these concepts can be considered (and were considered)
for any metric space. We restrict ourselves to Banach
spaces and concentrate on the class of {\it convex bodies}
$\mathcal{C}$ or one of the following subclasses of $\mathcal{C}$:

{\it Half-spaces} $\mathcal{H}_F$, i.e. sets of the form
$H_{x^*t}=\{x\in X:x^*(x)\le t\}\,$, $\;x^*\in F$, $x^*\ne 0\,$,
$\;t\in\mathbb{R}$, where $F$ is a subset of $X^*$. We denote
$\mathcal{H}_{X^*}=\mathcal{H}$, for short. Obviously,
$\partial_\delta H_{x^*t}=\{x\in X:|x^*(x)-t|<\delta\}$.

{\it Balls} $\mathcal{B}$, and {\it balls with radii $\le 1\,$}
(small balls) $\mathcal{B}_1$. Obviously, $\;\partial_{\delta}
B_r(z)= \{x\in X:r-\delta<\|x-z\|<r+\delta\}$.

Study of balls with radii $\le r$ can be reduced to
$\mathcal{B}_1$ by introducing of a proportional norm. The
$\mathcal{B}_1$-ideality is interesting also for investigation
of the Vitali theorem in Banach spaces \cite[p.~148]{Topsoe3}.
Study of the $\mu$-continuity for balls leads to questions on
geometry of unit sphere which are interesting by itself.
Necessary definitions and results in Banach space theory can
be found in \cite{BenyaminiLindenstrauss}, \cite{DevilleGZ},
\cite{HajekMVZ}. However we recall, for the convenience of
the reader, relevant definitions.
\medskip

We know only papers \cite{Sazonov}, \cite{Topsoe3},
\cite{TopsoeDudleyHoffmannJorgensen}, \cite{Aniszczyk} concerning
$\mu$-continuity directly. Let us recall some results of the
mentioned papers (mainly of \cite{Topsoe3}), simultaneously
presenting relevant statements of our note. \medskip

The situation is rather clear for the class $\mathcal{H}$. Every
finite-dimensional normed space $X$ is $\mathcal{H}$-ideal (see
e.g. \cite[Theorem~2]{Topsoe2}). No infinite-dimensional
Banach space is $\mathcal{H}$-ideal (it follows from \cite{Sazonov}).
We obtain this result from a simple and mostly known

\begin{prop}
\label{halfspace}
Let $X$ be an infinite-dimensional Banach space and $F\subset X^*$
be a total linear subspace. Then the class $\mathcal{H}_F$ is
uniformly discontinuous.
\end{prop}

A relation between the $\mu$-continuity of the classes of balls
and half-spaces in $c_0$ was considered in \cite{Topsoe3}. Namely,
let $\mathcal{H}_N$ be the class of half-spaces of the form
$H_{nt}=\{(a_1,a_2,\dots)\in c_0:a_n\le t\}\,$, $n\in\mathbb{N}$,
$\;t\in \mathbb{R}$. In the space $c_0$, the class $\mathcal{B}$
is $\mu$-continuous if and only if $\mathcal{H}_N$ is
$\mu$-continuous. We give an abstract version of this result and
show that every Banach space $X$ admits a measure $\mu$ for which
$\mathcal{B}$ is $\mu$-continuous, but $\mathcal{H}$ is not.
\smallskip

Only a few results about $\mathcal{B}$-ideal spaces were known.
Tops{\o}e \cite[p.~149]{Topsoe3} writes ``(practically) never
Banach space is $\mathcal{B}$-ideal; just consider a measure
concentrated on a line in $\mathbb{R}^2$''. This idea is realized
in the following statement.

\begin{cor}\label{smoothConvex}
No smooth and no rotund $($infinite- or finite-dimensional$)$
Banach space is $\mathcal{B}$-ideal.
\end{cor}

Further, $\mathcal{B}$ is uniformly discontinuous in the space
$\ell_p$ of $p$-summable sequences, $1\le p<\infty$
\cite[p.~155]{Topsoe3}. We generalize a part of this fact for
$p>1$.

\begin{cor}\label{smooth}
The class $\mathcal{B}$ is uniformly discontinuous in every
infinite-dimensional Banach space $X$ having a smooth norm.
\end{cor}

The only known $\mathcal{B}$-ideal space is the space $C(S)$ of
continuous functions over a compact metric set $S$. More precisely,
using some reasoning of Tops{\o}e \cite[p.~285--286]{Topsoe1},
Aniszczyk proved \cite{Aniszczyk} that $C(S)$ is
$\mathcal{B}_1$-ideal and claimed that his proof implies the
$\mathcal{B}$-ideality of $C(S)$.
\medskip

As for $\mathcal{B}_1$-ideality, every finite-dimensional normed
space is $\mathcal{B}_1$-ideal \cite[p.~153]{Topsoe3}. The spaces
$\ell_p,\;1\le p<\infty$, are $\mathcal{B}_1$-ideal, but $c_0$ is
not \cite[p.~154]{Topsoe3}. Moreover, the class $\mathcal{B}_1$ is
uniformly discontinuous for $c_0$ \cite[p.~151]{Topsoe3}.
Similarly, the space $L_1$ of absolutely integrable on $[0,1]$
functions is not $\mathcal{B}_1$-ideal
\cite[p.~144]{TopsoeDudleyHoffmannJorgensen}. The authors of
\cite{TopsoeDudleyHoffmannJorgensen} conjectured that
$\mathcal{B}_1$ is uniformly discontinuous in $L_1$.

We show that every infinite-dimensional Banach space can be
(equivalently) renormed in such a way that $\mathcal{B}_1$
becomes uniformly discontinuous in the new norm. However, it is not
known, whether every Banach space $X$ can be renormed so that $X$
becomes $\mathcal{B}_1$-ideal in the new norm.

Tops{\o}e \cite[p.~157]{Topsoe3} asked whether every
subspace of $\ell_p$, $1\le p<\infty$, is $\mathcal{B}_1$-ideal?
The almost positive answer follows from the next statement (see
Corollary \ref{ell-p} below for details).

\begin{cor}\label{weak-p}
Every dual Banach space with property $(\mathbf{m}^*)$ is
$\mathcal{B}_1$-ideal.
\end{cor}

See Section 2 for the definition of property $(\mathbf{m}^*)$.
Here we note only that it is similar to Kalton's property
$(\mathbf{M})$ from \cite{Kalton}. \medskip

Proofs of this paper are simple and geometrically clear. We try to
single out geometrical properties of Banach spaces ``responsible''
for various properties of $\mu$-continuity.


\section{Main results}

First we make some general remarks. Intersections of the form
$\bigcap_n\partial_{\delta_n}U_n$, where $U_n\in \mathcal{U}$ and
$\delta_n\downarrow 0$, play a decisive role for the establishing
of uniform $\mu$-continuity. We call these intersections
{\it Tops{\o}e $\mathcal{U}$-sets}.

\begin{lem}\label{TopsoeLemma}
{\rm \cite[p.~151]{Topsoe1}}. A class $\mathcal{U}$ is uniformly
$\mu$-continuous if and only if $\mu(A)=0$ for each Tops{\o}e
$\mathcal{U}$-set $A$.
\end{lem}

We will use well known

\begin{lem}\label{lemmaL}
Let $\mu$ be a measure in a Banach space $X$ and let
$\varepsilon,\delta>0$. Then there exists a finite-dimensional
subspace $E\subset X$ so that $\mu(E_\delta)>1-\varepsilon$,
where $E_\delta$ is the $\delta$-neighborhood of $E$.
\end{lem}

\begin{lem}\label{shift}
Let a class $\mathcal{U}$ be shift invariant in a Banach space
$X$, i.e. $U+x\in\mathcal{U}$ for all $U\in\mathcal{U}$ and $x\in
X$. Then $\mathcal{U}$ is $\mu$-continuous $($uniformly
$\mu$-continuous$)$ if and only if it is $\mu_x$-continuous
$($resp. uniformly $\mu_x$-continuous$)$, where
$$\mu_x(U):= \mu(U+x)\,,\;\;U\in\mathcal{U}\,,\;\;x\in X.$$
\end{lem}

Most classes and properties we consider are shift invariant. We
often use this fact, without mentioning it explicitly. Using
Lemma \ref{shift}, we also talk about ``shifting of a picture''.
\medskip


\textbf{\emph{1. Classes $\mathcal{C}$ and $\mathcal{H}_F$.}} The
class of convex sets in $\mathbb{R}^n$ is an important object of
study in the theory of empirical distributions. Since each measure is
almost concentrated on a compact set (the Ulam theorem) and a
compact set in an infinite-dimensional Banach space coincides with its
boundary, the consideration of the class $\mathcal{C}$ of convex
bodies, i.e. of all convex closed sets having interior points,
is natural. Formally, if there is no measure $\mu$, respect to
which $\mathcal{U}$ is $\mu$-continuous, the space $X$ is assumed
to be $\mathcal{U}$-ideal too. So, to avoid this uninteresting
situation, one wishes to be sure in the existence of at least one
such measure. The following (in fact, well known) proposition
guarantees the existence of the mentioned measure for convex bodies.
\medskip

Recall that a subset $D$ of a Banach space $X$ is called
{\it directionally porous} \cite[p.~166]{BenyaminiLindenstrauss}
if there exists a $0<\lambda<1$ so that for every $x\in D$ there
is a direction $e\in S_X$, and points $x_n=x+\varepsilon_ne$ with
$\varepsilon_n\to 0$, for which $D$ intersects no ball
$B_{\lambda\varepsilon_n}(x_n)$. A \emph{hyperplane} in a Banach
space $X$ is a shift of a closed one-codimensional subspace by
some vector.

\begin{prop}\label{body}
The class $\mathcal{C}$ is $\mu$-continuous for every
non-degenerated $($i.e. non-concentrated on any hyperplane$)$
Gaussian measure $\mu$.
\end{prop}

\begin{proof} Simple application of the Hahn-Banach theorem
and the Riesz lemma shows that the boundary of any convex
body is directionally porous with an arbitrary
$\lambda\subset (0,1)$. But for sets with directionally porous
boundaries, the property of Proposition \ref{body} holds
(see e.g. \cite[p.~167]{BenyaminiLindenstrauss}).
\end{proof}

\begin{remark}
Simple examples show that the non-degeneracy of a (non-Gaussian)
measure does not guarantee $\mu$-continuity, even for the class
$\mathcal{B}$.
\end{remark}

Let us pass to the class $\mathcal{H}_F$. A subspace $F\subset X^*$
is said to be {\it total} if for every $x\in S_X$ there is a
functional $x^*\in F$ such that $x^*(x)\ne 0$. \medskip

\noindent\emph{Proof of Proposition} \ref{halfspace}. Suppose, on the
contrary, that $\mathcal{H}_F$ is uniformly $\mu$-continuous for some
measure $\mu$. Every finite-dimensional subspace $E\subset X$
is closed in the weak topology $w(X,F)$ for a total subspace $F$. So,
by the Hahn-Banach theorem, $F$ is contained in a hyperplane of
a form $\{x\in X:x^*(x)=0\}\,$, $\;x^*\in F$. This hyperplane is, of
course, a boundary of the half-space $H_{x^*0}$. Therefore, by
the definition of uniform $\mu$-continuity, $\forall\,\varepsilon>0$
there is a $\delta>0$ such that $\mu(E_\delta)<\varepsilon$ for each
finite-dimensional subspace $E\subset X$, where $E_\delta$ is the
$\delta$-neighborhood of $E$. But $\mu(E_\delta)>1-\varepsilon$
for the subspace $E$ from Lemma \ref{lemmaL}. If $\varepsilon<\frac{1}{2}$,
we get a contradiction. \qed
\medskip

For spaces $\ell_p$, $1\le p<\infty$, and $F=X^*$ Proposition~
\ref{halfspace} is a part of \cite[Proposition~3]{Topsoe3}.

\begin{cor}\label{H-ideal}
No infinite-dimensional Banach space $X$ is $\mathcal{H}_F$-ideal
for any total linear subspace $F\subset X^*$.
\end{cor}

\begin{proof}
According to Proposition \ref{body}, every half-space is
$\mu$-continuous for each non-degenerated Gaussian measure $\mu$
on $X$. However, by Proposition \ref{halfspace}, the class
$\mathcal{H}_F$ is not uniformly $\mu$-continuous. Therefore, $X$
is not $\mathcal{H}_F$-ideal.
\end{proof}

In connection with Corollary \ref{H-ideal}, it is interesting to
localize natural measures $\mu$ and uniformly $\mu$-continuous
subclasses $\mathcal{U}\subset \mathcal{H}$. Two such examples
were presented in \cite{Sazonov} and \cite{MatsakPlichko}; we will
present one more. The following statement shows a relation between
the uniform $\mu$-continuity of balls and the uniform
$\mu$-continuity of half-spaces whose boundaries are tangent to
these balls. Let us recall necessary definitions.
\medskip

A hyperplane $D$ is {\it tangent} to a ball $B$ at a point $x$ if
$x\in D\cap B\subset \partial B$. A point $x\in
\partial B$ is called a {\it smooth point} of $B$ if for every
$y\in X$ there exists $\lim_{\lambda\to 0}\lambda^{-1}(\|x+\lambda y\|-\|x\|)$
\cite[p.~409]{BenyaminiLindenstrauss}.
Geometrically it means that if (for example) a hyperplane $D$
touches a ball $B_1(z)$ ($\|z\|=1$) at $0$ and $x\in D$ then the
distance $\mathrm{dist}(x,\partial B_k(kz))\to 0$ as $k\to\infty$.
A norm of a Banach space is called \emph{smooth} if each
$x\in \partial B$ is a smooth point.
\medskip

Denote by $\mathcal{H}_S$ the class of half-spaces whose
boundaries are tangent to balls of a Banach space $X$ at smooth
points of these balls.
\begin{prop}\label{tangent}
If the class of balls $\mathcal{B}$ in a Banach space $X$ is
uniformly $\mu$-continuous then the class of half-spaces
$\mathcal{H}_S$ is uniformly $\mu$-continuous too.
\end{prop}
\begin{proof} Suppose that $\mathcal{H}_S$ is not
$\mu$-continuous. Then there exists $\varepsilon>0$ such that
$\forall\,\delta>0$ there is a half-space $H\in \mathcal{H}_S$
(depending on $\delta$) with $\mu(\partial_\delta H)>\varepsilon$.
Since $X$ is separable, there exists a finite collection of balls
$(K_i)_1^n$ (depending on $\delta$), with centers
$(z_i)_1^n\subset \partial H$, each of radius $<\delta$, such that
$\mu(\cup_1^nK_i)>\varepsilon$. Let $B_1(z)$ be a ball for which
$\partial H$ is tangent and let $x$ be a point of tangency.
Shifting, by Lemma \ref{shift}, the whole picture on $-x$, one may
assume $x=0$. Since $x=0$ is a smooth point of $B_1(z)$, for all
$i$
$$\mathrm{dist}(z_i,\partial B_k(kz))\to 0\;\;\;\mbox{as}\;\;\;k\to\infty.$$
Hence, for sufficiently large $k\,$,
$$\cup_1^nK_i\subset \partial_\delta B_k(kz),\;\;\mbox{so}\;\;
\mu(\partial_\delta B_k(kz))>\varepsilon.$$ This contradicts the
uniform $\mu$-continuity of $\mathcal{B}$.
\end{proof}

\begin{remark}
This proof uses an idea from \cite[Lemma 3]{Szabados}, where
it was proved that the uniform $\mu$-continuity of balls implies the
uniform $\mu$-continuity of half-spaces in a Hilbert space.
I would like to mention that the word ``uniform'' is missing in the
statement of \cite[Lemma 3]{Szabados}. Without this word the
statement is false, an easy example can be constructed even in
two-dimensional Euclidean space. We will return below to relations
between the $\mu$-continuity of balls and half-spaces.
\end{remark}

From now on by $\mathcal{H}_S$ we denote the class of half-spaces
of the form $H_{st}=\{x\in C(S): x(s)\le t\}\,$, $s\in S\,$, $t\in
\mathbb{R}$. By argumentation of the following remark, for $C(S)$
this ``new'' class $\mathcal{H}_S$ is contained in the ``old one''.

\begin{remark} \label{C(K)-ideal}
If the above mentioned result of Aniszczyk is true then the space
$C(S)$ is $\mathcal{H}_S$-ideal.
\end{remark}

Indeed, each half-space $H_{st}$ touches the ball $B_t(0)$ at a
smooth point -- a continuous function whose modulus attains
maximum at the unique point $s$. Since, by Aniszczyk's result
\cite{Aniszczyk}, the space $C(S)$ is $\mathcal{B}$-ideal,
Proposition \ref{tangent} implies the $\mathcal{H}_S$-ideality
of $C(S)$.

\begin{remark}
Remark \ref{C(K)-ideal} together with Corollary \ref{H-ideal}
show that a space $X$ can be $\mathcal{H}_F$-ideal but not
$\mathcal{H}_{\mathrm{lin}F}$-ideal.
\end{remark}
We return to relations between the $\mu$-continuity of balls and
half-spaces. Let $X$ be a Banach space. A point $x$ of an
arbitrary ball $B\subset X$ is called {\it exposed} if there is a
functional $x^*\in X^*$ so that $x^*(x)>x^*(y)$ for all $y\in B$,
$y\ne x$ \cite[p.~108]{BenyaminiLindenstrauss}.

\begin{prop}\label{cobtinuous-not}
Any Banach space $X$ admits a measure $\mu$ for which the class
$\mathcal{B}$ is $\mu$-continuous, but $\mathcal{H}$ is not.
\end{prop}

\begin{proof} If $X$ is reflexive then its unit ball contains an
exposed point \cite[p.~110]{BenyaminiLindenstrauss}, i.e. there
exists a point $x\in S_X$ and a hyperplane $D$ such that $D\cap
S_X=\{x\}$. If $X$ is not reflexive then, by the James theorem,
there exists a functional $x^*\in X^*$ which does not attain its
norm \cite[p.~14]{DevilleGZ}. In both cases there is a hyperplane
$D$ such that for every ball $B$, the intersection $D\cap B$ either
is empty or contains a single point or is a convex body in $D$.

Take a Gaussian measure $\mu$ which is concentrated and
non-generated on $D$. By Proposition \ref{body}, the class
$\mathcal{B}$ is $\mu$-continuous and, by Proposition
\ref{halfspace}, $\mathcal{H}$ is not. \end{proof}

The above mentioned result of Tops{\o}e \cite[p.~152]{Topsoe3}
shows, that the converse to the previous statement is false. In
$c_0$, even the $\mu$-continuity of the subclass
$\mathcal{H}_N\subset\mathcal{H}$ already implies the
$\mu$-continuity of $\mathcal{B}$ (see Section 1). However, under
additional conditions a converse to Proposition
\ref{cobtinuous-not} is valid. A Banach space $X$ is called
\emph{rotund} if $\|x+y\|<\|x\|+\|y\|$ for all linearly
independent elements $x,y\in X$.

\begin{prop}
Every rotund Banach space $X$ admits a measure $\nu$ for which the
class $\mathcal{H}$ is $\nu$-continuous, but $\mathcal{B}$ is not.
\end{prop}

\begin{proof} Let $\mu$ be a non-degenerated Gaussian measure on
$X$. Given arbitrary Borel subsets $A\subset S_X$ and $U\subset X$, put
$$\nu(A):=\mu\{tx: x\in A\,,\;t\ge 0\}\;\;\;\mbox{and}\;\;\;
\nu(U):=\nu(U\cap S_X).$$

Of course, $\nu$ is a measure on $X$ and $\nu(S_X)=1$.
Since $X$ is rotund, given a hyperplane $D$, $D\cap S_X$ is either
empty or contains a single point, or is a boundary of the convex body
$D\cap B_1(0)$ in $D$. The first and second cases are not
interesting, and for the third
$$\nu(D)=\nu(D\cap S_X)=\mu\{tx: x\in D\cap S_X\,,\;t\ge 0\}.$$
However, the set $\{tx: x\in D\cap S_X\,,\;t\ge 0\}$ is a boundary
of the convex body $\{tx: x\in D\cap B_1(0)\,,\;t\ge 0\}$.
According to Proposition \ref{body}, $\mu$-measure of such boundary
equals to zero. So, the class $\mathcal{H}$ is $\nu$-continuous.
Obviously, $\mathcal{B}$ is not $\nu$-continuous.
\end{proof}

Now we present the promised abstract version of Tops{\o}e's
statement on $c_0$. A Banach space $X$ is called \emph{polyhedral}
\cite{FonfLP} if a ball of every of its finite-dimensional
subspace is a polyhedron.

\begin{prop}
For any polyhedral Banach space $X$ there is a countable subset
$F\subset S_{X^*}$ so that the class $\mathcal{B}$ is
$\mu$-continuous if and only if $\mathcal{H}_F$ is.
\end{prop}

\begin{proof} Put $D_{x^*}=\{x\in X:x^*(x)=1\}$. In view of the
Fonf theorem on the structure of a sphere in a polyhedral space
\cite[p.~655]{FonfLP}, there is a countable subset $F\subset
S_{X^*}$ so that:
\begin{enumerate}
\item[\sf $(a)$] $S_X=\bigcup_{x^*\in F}(S_X\cap D_{x^*})$ and\smallskip

\item[\sf $(b)$] each set $S_X\cap D_{x^*}$, $x^*\in F$, has an
interior point in $D_{x^*}$.
\end{enumerate}

Now, let $\mu(\partial B_r(z))>0$ for some $r$ and $z$. Then, by $(a)$,
there are $x^*\in F$ and $t\in \mathbb{R}$ such that
$$\mu\{\partial H_{x^*t}\cap \partial B_r(z)\}>0$$
hence, $\mu(\partial H_{x^*t})>0$.

Conversely, let $\mu(\partial H_{x^*t})>0$ for some $x^*\in F$ and
$t\in \mathbb{R}$. By $(b)$, for some $r$ and $z$ the intersection
$\partial H_{x^*t}\cap \partial B_r(z)$ has an interior point in
$\partial H_{x^*t}$. Then
$$\mu\{\partial H_{x^*t}\cap \partial B_{r}(z')\}>0$$
for a translate $B_r(z')$ of the ball $B_r(z)$. Hence,
$\mu(\partial B_{r}(z'))>0.$
\end{proof}


\textbf{\emph{2. Class $\mathcal{B}$}}.\medskip

\noindent\emph{Proof of Corollary} \ref{smooth}. In fact, if the
class $\mathcal{B}$ was uniformly $\mu$-continuous then, according
to Proposition \ref{tangent}, $\mathcal{H}$ would be uniformly
$\mu$-continuous too. This contradicts Proposition \ref{halfspace}.
\qed \medskip

Corollary \ref{smooth} implies the uniform discontinuity of
$\mathcal{B}$ in the spaces $\ell_p$ and $L_p$ for $1<p<\infty$.

\begin{prop}\label{Mazur}
Every $($infinite- or finite-dimensional$)$ Banach space $X$
admits a $($degenerated$)$ Gaussian measure $\mu$ for which the
class $\mathcal{B}$ is not uniformly $\mu$-continuous.
\end{prop}
\begin{proof} According to the Mazur theorem
\cite[p.~91]{BenyaminiLindenstrauss}, each ball $B_1(z)$ of $X$
has a smooth point $x$ (of course, we assume $\dim X>1$). Let $D$
be a hyperplane in $X$, tangent the ball $B_1(z)$ at a point $x$.
Take a Gaussian measure concentrated on $D$. Below we repeat the
arguments  of Proposition \ref{tangent}. Namely, given
$0<\varepsilon<1$, for every $\delta>0$ there exist balls
$(K_i)_1^n$ of $D$ (depending on $\delta$), with centers
$(z_i)_1^n$, each of radius $<\delta$, such that
$\mu(\cup_1^nK_i)>\varepsilon$. Shifting, by Lemma \ref{shift},
the whole picture on $-x$, one may assume $x=0$ (then $\|z\|=1$).
Since $B_1(z)$ is smooth at point $0$, the distance
$\mathrm{dist}(z_i,\partial B_k(kz))\to 0$ as $k\to\infty$, for
all $i$. Hence, for sufficiently large $k\,$,
$\;\cup_1^nK_i\subset
\partial_\delta B_k(kz)$, so $\mu(\partial_\delta
B_k(kz))>\varepsilon$. Therefore, $\mathcal{B}$ is not uniformly
$\mu$-continuous. \end{proof}

\begin{prop}\label{extreme-smooth}
Suppose the unit ball of a Banach space $X$ contains a point $x$ which
is exposed and smooth simultaneously. Then $X$ is not
$\mathcal{B}$-ideal.
\end{prop}

\begin{proof} Let $D$ be a hyperplane, tangent to the unit ball of
$X$ at the point $x$. Take a measure $\mu$ concentrated and
non-degenerated on $D$. By Proposition \ref{Mazur}, the class
$\mathcal{B}$ is not uniformly $\mu$-continuous. Let $B$ be an
arbitrary ball in $X$. Since $x$ is exposed, $B\cap D$ (provided
it is nonempty and does not consist of a single point) is a convex
body in $D$, and $\partial B\cap D=\partial(B\cap D)$. By
Proposition \ref{body}, $\mu(\partial(B\cap D))=0$. Hence, $\mu$
vanished on each sphere of $X$, so $\mathcal{B}$ is
$\mu$-continuous. Therefore $X$ is not $\mathcal{B}$-ideal.
\end{proof}

Note that every Banach space $X$ can be renormed so that in the
new norm the sets of exposed and smooth points are disjoint.
To verify this, one can introduce first a smooth norm in $X$
\cite[p.~89]{BenyaminiLindenstrauss}, and then the new norm as
in (\ref{l-sum}) below. \medskip

\noindent\emph{Proof of Corollary} \ref{smoothConvex}. Consider
three cases.

1. \emph{The space $X$ is smooth and infinite-dimensional.} By
virtue of Proposition \ref{body}, each ball of $X$ is
$\mu$-continuous for each non-degenerated Gaussian measure $\mu$.
On the other hand, $\mathcal{B}$ is uniformly discontinuous, by
Corollary \ref{smooth}. Hence, $X$ is not $\mathcal{B}$-ideal.

2. \emph{$X$ is smooth and finite-dimensional.} Then $S_X$ has
an exposed point \cite[p.~110]{BenyaminiLindenstrauss} which, as
all others, is a smooth point. It remains to apply Proposition
\ref{extreme-smooth}.

3. \emph{$X$ is rotund.} According to the Hahn-Banach theorem,
every point of $S_X$ is exposed. By the mentioned Mazur theorem,
$S_X$ contains a smooth point. We apply Proposition
\ref{extreme-smooth} once more. \qed \medskip

The following statement shows that Corollary \ref{smooth} is not
valid without additional assumptions.

\begin{prop}
Suppose a Banach space has the form $X=Y\oplus L$, where $Y$ is a
closed subspace, $L$ is a one-dimensional subspace, and
\begin{equation}\label{l-sum}
\|y+l\|=\|y\|+\|l\|\,,\;\;\;y\in Y\,,\;\;\;l\in L. \end{equation}
Let $\mu$ be a one-dimensional Gaussian measure on $L$. Then the
class $\mathcal{B}$ is uniformly $\mu$-continuous in $X$.
\end{prop}

\begin{proof} Take an arbitrary ball $B_r(z)$ of $X$. Then
$B_r(z)\cap L\subset \mathrm{lin}(z,L)$, moreover, $\mathrm{lin}(z,L)$
is isometric to the two-dimensional $\ell_1^2$ or to a
one-dimensional space. Hence, the sphere $\partial B_r(z)$
intersects $L$ at at most two points and the length
(in the norm of $X$) of $\partial_{\delta}B\cap L$ is not
greater than $4\delta$. Therefore, $\mu(\partial_{\delta}B)\to 0$
as $\delta\to 0$, uniformly on the balls. \end{proof}

\begin{cor}
Every Banach space $X$ can be renormed so that
in the new norm the class $\mathcal{B}$ becomes uniformly
$\mu$-continuous for some $($degenerated$)$ measure
$\mu$.
\end{cor}

\begin{proof} Let  $Y\subset X$ be a one-codimensional closed
subspace, $L\subset X$ be a one-dimensional subspace and $L\cap
Y=0$. The desired norm can be introduced by (\ref{l-sum}). \end{proof}

We are unaware of publications on $\mathcal{B}$-ideal
finite-dimensional spaces. In view of Corollary
\ref{smoothConvex}, a $\mathcal{B}$-ideal space cannot be smooth
or rotund. On the other hand, by Aniszczyk's theorem, the
$n$-dimensional space $\ell_{\infty}^n$ is $\mathcal{B}$-ideal.
One can advance, as a working hypothesis, that a
finite-dimensional normed space is $\mathcal{B}$-ideal if and only
if it is polyhedral.
\medskip


\textbf{\emph{3. Class $\mathcal{B}_1$}}. We start with negative
results and single out a class of norms for which $\mathcal{B}_1$
is uniformly discontinuous.

\begin{definition}\label{DedFinUniv}
We say that a Banach space $X$ has a {\it finite universal sphere}
if there exists $r>0$ such that $\forall\,\delta>0$ and every
finite-dimensional subspace $E\subset X$ there is $z\in X$ so that
the ball $K_r^E=B_r(0)\cap E$ of $E$ belongs to
$\partial_{\delta}(B_1(z))$.
\end{definition}

\begin{remark} \label{RemarkUniversal}
Obviously, when we check that the condition of this definition
is satisfied for $X$, it is sufficient to consider $E$ from an
arbitrary increasing sequence $(E_n)$ of finite-dimensional
subspaces whose union is dense in $X$.
\end{remark}

\begin{remark}
Definition \ref{DedFinUniv} can be considered as an ``approximative
and uniform'' version of the following well known concept: A Banach
space $X$ \emph{contains no finite-dimensional Haar $($or \v{C}eby\v{s}ev$)$
subspaces} if for every finite-dimensional subspace $E\subset X$
and for each $x\in X\setminus E$ there are at least two best
approximations in $E$. For example, $L_1$ contains no finite-dimensional
Haar subspaces \cite[Theorem~2.5]{Phelps}.
\end{remark}

\begin{prop}\label{Universal}
The class $\mathcal{B}_1$ is uniformly discontinuous in any
Banach space $X$ with a finite universal sphere.
\end{prop}

\begin{proof}
Let $r$ be the constant from Definition \ref{DedFinUniv} and
$\mu$ be a measure on $X$. By Lemma \ref{shift}, one may
assume $\mu\{B_r(0)\}>\varepsilon$ for some $\varepsilon>0$.
According to Lemma \ref{lemmaL}, for every $\delta>0$ there
is a finite-dimensional subspace
$E\subset X$ so that $\mu(E_{\delta})>1-\frac{\varepsilon}{2}$.
Then
$$\mu\{B_r(0)\setminus E_{\delta}\}\le \mu\{X\setminus E_{\delta}\}\le
1-(1-{\textstyle\frac{\varepsilon}{2}})={\textstyle\frac{\varepsilon}{2}},$$
so, since $K_r^E=B_r(0)\cap E$,
$$\mu\{(K_r^E)_{\delta}\}\ge\mu\{B_r(0)\}-\mu\{B_r(0)\setminus E_{\delta}\}\ge
\varepsilon-{\textstyle\frac{\varepsilon}{2}}={\textstyle\frac{\varepsilon}{2}}.$$
By definition, for some $z$
$$K_r^E\subset \partial_{\delta}B_1(z),$$ whence
$$(K_r^E)_{\delta}\subset \partial_{2\delta}B_1(z).$$
Hence $\mu\{\partial_{2\delta}(B_1(z)\}>\frac{\varepsilon}{2}\,,$
i.e. $\mathcal{B}_1$ is not uniformly $\mu$-continuous.
\end{proof}

\begin{prop}\label{renorming}
Every infinite-dimensional Banach space $X$ can be renormed
so that the new sphere will be finite universal.
\end{prop}

\begin{proof} Let $(E_n)$ be an increasing sequence of
finite-dimensional subspaces whose union is dense in $X$. In
just the same way as in \cite[p.~7]{HajekMVZ}, one can show
the existence in $X$ of infinite Auerbach system, i.e. a
biorthogonal sequence $(x_n,x_n^*)\,$, $\,\|x_n\|=\|x_n^*\|=1$,
with an additional condition: for all $n$
$$x_n^*(x)=0\;\;\;\mbox{as soon as}\;\;\;x\in E_n.$$

The new norm on $X$ can be introduced by the formula
$$|\!|\!|x|\!|\!|=\max\left\{\|x\|,\,2\sup\nolimits_n|x_n^*(x)|\right\}.$$

We check that the new sphere is finite universal. Fixing $n$,
consider the set
$$K_n=\left\{x\in B_1(0):\; x_n^*(x)=
{\textstyle\frac{1}{2}},\;{\textstyle\|x-\frac{1}{2}}x_n\|<{\textstyle\frac{1}{2}}\right\}.$$

If $m\ne n$ and $x\in K_n$ then
$$|x_m^*(x)|=|x_m^*(x-{\textstyle\frac{1}{2}}x_n)|\le
\|x-{\textstyle\frac{1}{2}}x_n\|<{\textstyle\frac{1}{2}},$$ so
$|\!|\!|x|\!|\!|=1$. Moreover, if $e\in E_n$ and
$\|e\|<\frac{1}{2}$ then
$$x_n^*(e+{\textstyle\frac{1}{2}}x_n)={\textstyle\frac{1}{2}},\;\;
\|e+{\textstyle\frac{1}{2}}x_n\|\le 1\;\;
\mbox{and}\;\;\|(e+{\textstyle\frac{1}{2}}x_n)-{\textstyle\frac{1}{2}}x_n\|<
{\textstyle\frac{1}{2}}\,.$$

Therefore, $K^{E_n}_{1/2}$ belongs to the new sphere of radius $1$ with
center ${\textstyle\frac{1}{2}}x_n$. Applying Remark
\ref{RemarkUniversal} we get that the new sphere is finite
universal with $r=\frac{1}{2}$. \end{proof}

This construction of $|\!|\!|\;|\!|\!|$ is similar to a
construction in \cite{Orno}.

\begin{cor}
If a Banach space is $\mathcal{B}_1$-ideal with respect to any
equivalent norm then it is finite-dimensional. The
sphere of a finite-dimensional space cannot be finite universal.
\end{cor}

\begin{proof} The first part of corollary is a simple combination
of Propositions \ref{body}, \ref{Universal} and \ref{renorming}.
The second part follows from Proposition \ref{Universal} and
mentioned Tops{\o}e's result \cite[p.~153]{Topsoe3} (see also
Corollary \ref{finiteIdeal} below). \end{proof}

\noindent\emph{Example.} Let $R=\cup_nS_n$ be an increasing
sequence of metric compact sets $S_n$. The space $C_0(R)$ of
continuous function $x(s)$ with $x(s)\to 0$ as $s\to\infty$ has
finite universal sphere. In particular, the spaces $c_0$ and
$C_0(\mathbb{R})$ have finite universal sphere. \smallskip

The verification is simple. Similarly, one can easily  check that
every subspace of $c_0$ has finite universal sphere and that the
space $C(S)$, with a compact metric $S$, has not finite universal
sphere. It is not hard to prove the following

\begin{prop}\label{complements}
Suppose a Banach space $Y$ has finite universal sphere and
$X=Y\oplus Z$ with
$$\|y+z\|=\max(\|y\|,\|z\|),\;\;y\in Y,\;\;z\in Z.$$
Then $X$ has finite universal sphere too.
\end{prop}

Now we turn to positive results. First we consider a compact set
of centers. The following Corollary \ref{compact} is more or less
known (cf. with \cite[Theorem 6]{BillingsleyTopsoe}). We present
its direct and simple proof.

\begin{lem}\label{compactLemma}
Let $X$ be a Banach space. Every Tops{\o}e $\mathcal{B}$-set
$A=\bigcap_n\partial_{\delta_n}B_{r_n}(z_n)$, with a convergent
sequence of centers $(z_n)$, belongs to some
sphere of $X$.
\end{lem}

\begin{proof} We show that for some sequence $(n_k)$
the intersection
$A=\bigcap_k\partial_{\delta_{n_k}}B_{r_{n_k}}(z_{n_k})$
is a sphere. Let $z_n\to z$; hence $(z_n)$ is bounded and
the sequence $(r_n)$ is bounded too (otherwise $A$ would be empty).
Take a sequence $(n_k)$ so that $r_{n_k}\to r$ as
$k\to\infty$. Passing to a subsequence, one may assume
$(z_{n_k})$, $(r_{n_k})$ and $(\delta_{n_k})$ to be convergent
very quickly. More precisely, one may assume that for all $k$
$$\|z_{n_k}-z\|<2^{-k-4}\;,\;\;\;|r_{n_k}-r|<2^{-k-4}\;
\;\;\;\mbox{and}\;\;\;\delta_{n_k}<2^{-k}\,.$$ Now, we slightly
increase $\delta_{n_k}$ (namely, take $\delta_{n_k}=2^{-k}$).
By the triangle inequality, the obtained sequence of
``rings'' $\partial_{\delta_{n_k}}B_{r_{n_k}}(z_{n_k})$ is decreasing and
every such ring contains the sphere $S$ with center $z$ and radius
$r$. Hence
$$\bigcap\nolimits_k\partial_{\delta_{n_k}}B_{r_{n_k}}(z_{n_k})=S;$$
so $A\subset S$. \end{proof}

\begin{cor}\label{compact}
Let $X$ be a Banach space. Let $\mathcal{U}$ be a class
consisting of balls $B_r(z)\,,\;z\in Z\,,\;r\in\mathbb{R}$,
with a compact set of centers $Z$. Then $X$ is
$\mathcal{U}$-ideal.
\end{cor}

\begin{proof} Let the class $\mathcal{U}$ be $\mu$-continuous.
Since $Z$ is compact, by Lemma \ref{compactLemma}, each Tops{\o}e
$\mathcal{U}$-set belongs to the boundary $\partial U$ of some set
$U\in \mathcal{U}$. So, by Lemma \ref{TopsoeLemma}, $\mathcal{U}$
is uniformly $\mu$-continuous, hence $X$ is $\mathcal{U}$-ideal.
\end{proof}

\begin{cor}\label{finiteIdeal}
Every finite-dimensional normed space is $\mathcal{B}_1$-ideal.
\end{cor}

Besides, we already know this corollary.

\begin{definition}
We say that a dual Banach space $X=Y^*$ {\it has property}
$(\mathbf{m}^*)$ if for every weakly* null sequence $x_n\in X$
such that $\|x_n\|\to c$ as $n\to\infty$ there exists a strictly
increasing function $\varphi(t)\ge t\,$, $\;t\ge 0$ (depending on
$c$), so that for all $x\in X$
$$\lim\nolimits_n\|x+x_n\|= \varphi(\|x\|).$$
\end{definition}

This definition is inspired, on the one hand, by the proof of
$\mathcal{B}_1$-ideality of $\ell_p$ from \cite{Topsoe3}, and on
the other hand, by the following well-known concept. A Banach
space $Y$ has {\it property} $(\mathbf{M}^*)$ of Kalton
\cite{Kalton} if for all elements $x,y\in X=Y^*$ with
$\|x\|=\|y\|$ and every weakly* null sequence $(x_n)$
$$\limsup\nolimits_n\|x+x_n\|= \limsup\nolimits_n\|y+x_n\|.$$ We
suspect that there is a connection between these properties.

\begin{remark}
Every space $\ell_p$, $1\le p<\infty$, has property
$(\mathbf{m}^*)$. There are spaces, different
from $\ell_p$, which have property $(\mathbf{m}^*)$. Such
property have, for example, the James spaces $J_p$, $1<p<\infty$.
The property $(\mathbf{m}^*)$ is hereditary: every weakly*
closed infinite-dimensional subspace of a space with
property $(\mathbf{m}^*)$ has this property.
\end{remark}

The following lemma and Corollary \ref{weak-p} generalize
\cite[Theorem~2]{Topsoe3}.

\begin{lem}\label{lweak-p}
Suppose a dual Banach space $X=Y^*$ has property $(\mathbf{m}^*)$.
Then every Tops{\o}e $\mathcal{B}_1$-set
$A=\bigcap_n\partial_{\delta_n}B_{r_n}(z_n)$ of $X$ belongs to a
sphere of $X$ with radius $\le 1$.
\end{lem}

\begin{proof} Without loss of generality one may assume the
set $(z_n)$ to be bounded, otherwise the intersection $A$ would be
empty. Take a sequence $(n_k)$ so that $r_{n_k}\to r\le 1$.
By weak* compactness, passing to a subsequence, one may assume
$z_{n_k}=z+x_{n_k}$ with weakly* null $(x_{n_k})$ and convergent
$(\|x_{n_k}\|)$. Then for all
$x\in\bigcap_k\partial_{\delta_{n_k}}B_{r_{n_k}}(z_{n_k})$
$$\|x-z-x_{n_k}\|=\|x-z_{n_k}\|\to r\;\;\;\mbox{as}\;\;\;k\to\infty.$$

Since $X$ has property $(\mathbf{m}^*)$,
$$\|x-z-x_{n_k}\|\to \varphi(\|x-z\|);\;\;\mbox{as}\;\;\;k\to\infty.$$

Hence, for all $x\in\bigcap_k\partial_{\delta_{n_k}}B_{r_{n_k}}(z_{n_k})$
$$ \varphi(\|x-z\|)=r,\;\;\;\;\mbox{whence}\;\;\;\; \|x-z\|=\varphi^{-1}(r).$$

So, $A$ belongs to the sphere with the center $z$ and radius
$\varphi^{-1}(r)\le r\le 1$. \end{proof}

\noindent\emph{Proof of Corollary} \ref{weak-p}. The contrary
means the existence of a measure $\mu$ respect to which
$\mathcal{B}_1$ is continuous, scalars $\varepsilon$,
$\delta_n\to 0$ and a sequence of balls $B_{r_n}(z_n)$, for
which $\mu\{\bigcap_n\partial_{\delta_n}B_{r_n}(z_n)\}>\varepsilon$.
This contradicts Lemma \ref{lweak-p}. \qed

\begin{remark}
The space $L_p$, $1<p<\infty$, $p\ne 2$, has not property
$(\mathbf{m}^*)$. Moreover Elena Riss has noted that for the
Rademacher functions $(\mathbf{r}_n)$ and $\delta_n\downarrow 0$
the Tops{\o}e set $\bigcap_n \partial_{\delta_n}B_1(\mathbf{r}_n)$
coincides with the set
$$E=\{x\in L_p:\textstyle{\frac{1}{2}}\|x-1\|^p+\textstyle{\frac{1}{2}}\|x+1\|^p=1\}.$$
The set $E$, for $p\ne 2$, belongs to no sphere of $L_p$.
\end{remark}

\begin{cor}\label{ell-p} {\rm (of Corollary \ref{weak-p}.)}
Every weakly* closed subspace of $\ell_p$, $1\le p<\infty$,
is $\mathcal{B}_1$-ideal.
\end{cor}

Let us recall, Tops{\o}e \cite[p.~157]{Topsoe3} asked whether
every subspace of $\ell_p$, $1\le p<\infty$, is
$\mathcal{B}_1$-ideal? Corollary \ref{ell-p} provides the positive
answer for $1<p<\infty$ and for all weakly* closed subspaces of
$\ell_1$.

\proof[Acknowledgements] The author wishes to express his thanks
to M.~Ostrovskii, I.~Matsak and L.~Zaj\'{\i}\v{c}ek  for valuable
consultations and to the referee for very useful comments and
suggestions. 
\end{document}